\documentclass[a4paper,10pt]{amsart}
\usepackage[utf8]{inputenc}
\usepackage{amsxtra}
\usepackage{amsopn}
\usepackage{amsmath,amsthm,amssymb}
\usepackage{amscd}
\usepackage{amsfonts}
\usepackage{latexsym}
\usepackage{verbatim}
\usepackage{color}
\usepackage{tikz-cd}

\theoremstyle{plain}
\newtheorem{theorem}{Theorem}[section]
\newtheorem*{theorem*}{Theorem}

\newtheorem{prop}[theorem]{Proposition}
\newtheorem{cor}[theorem]{Corollary}
\newtheorem{rem}[theorem]{Remark}
\newtheorem{ex}[theorem]{Example}
\newtheorem*{mt*}{Main Theorem}


\newcommand\C{{\mathbb C}}

\newcommand\Z{{\mathbb Z}}

\newcommand{\del}{\partial}
\newcommand{\delbar}{\bar{\del}}

\DeclareMathOperator{\Ker}{Ker}
\renewcommand{\H}{\mathcal{H}}
\DeclareMathOperator{\cinf}{\mathcal{C}^\infty}
\let\phi\varphi
\DeclareMathOperator{\End}{End}
\DeclareMathOperator{\id}{id}
\newcommand{\lr}{\longrightarrow}
\DeclareMathOperator{\vol}{Vol}
\let\c\overline
\let\b\bar

\title[Bott-Chern harmonic forms and primitive decompositions 
]{Bott-Chern harmonic forms and primitive decompositions on compact almost K\"ahler manifolds}

\author[Riccardo Piovani and Nicoletta Tardini]{Riccardo Piovani and Nicoletta Tardini}
\address{Dipartimento di Scienze Matematiche, Fisiche e Informatiche\\
Unit\`{a} di Matematica e Informatica,
Universit\`{a} degli Studi di Parma\\
Parco Area delle Scienze 53/A, 43124 \\
Parma, Italy}
\email{riccardopiovani1994@gmail.com}
\email{riccardo.piovani@unipr.it}
\email{nicoletta.tardini@gmail.com}
\email{nicoletta.tardini@unipr.it}

\thanks{\newline 
The first author is partially supported by GNSAGA of INdAM.
The second author is partially supported by GNSAGA of INdAM and has financially been supported by the Programme ``FIL-Quota Incentivante'' of University of Parma and co-sponsored by Fondazione Cariparma.}
\begin{document}

\maketitle

\begin{abstract}
Let $(X,J,\omega)$ be a compact $2n$-dimensional almost K\"ahler manifold. We prove primitive decompositions for Bott-Chern and Aeppli harmonic forms in special bidegrees and show that such bidegrees are optimal. We also show how the spaces of primitive Bott-Chern, Aeppli, Dolbeault and $\del$-harmonic forms on $(X,J,\omega)$ are related.
\end{abstract}

\section{Introduction}

Let $(X,J,\omega)$ be an almost Hermitian manifold of real dimension $2n$. Denote by 
\[
L:\Lambda^kX\to\Lambda^{k+2}X,\quad\alpha\mapsto\omega\wedge\alpha
\]
the Lefschetz operator, and by
\[
\Lambda:\Lambda^kX\to\Lambda^{k-2}X,\quad \Lambda:=-*L*
\]
its dual, where $*:\Lambda^kX\to\Lambda^{2n-k}$ is the Hodge $*$ operator.
A $k$-form $\alpha\in\Lambda^kX$, for $k\leq n$, is said to be {\em primitive} if $\Lambda\alpha=0$, or equivalently if $L^{n-k+1}\alpha=0$. Given that, the following vector bundle decomposition holds  
\begin{equation}\label{eq-prim-dec-forms-intro}
\Lambda^kX=\bigoplus_{r\geq\max(k-n,0)}L^r(P^{k-2r}X),
\end{equation}
where we denoted by
\[
P^{s}X:=\ker\big(\Lambda:\Lambda^{s}X\to\Lambda^{s-2}X\big)
\]
the bundle of primitive $s$-forms. The operators $L$ and $\Lambda$ extend to smooth sections, in particular to smooth $k$-forms $A^k:=\Gamma(X,\Lambda^kX)$ and to smooth $(p,q)$-forms $A^{p,q}:=\Gamma(X,\Lambda^{p,q}X)$. We also set $P^s:=\Gamma(X,P^sX)$ and $P^{p,q}:=\Gamma(X,P^{p,q}X)$, where $P^{p,q}X:=P^{p+q}X\cap\Lambda^{p,q}X$ is the bundle of primitive $(p,q)$-forms.

If $(X,J,\omega)$ is a compact K\"ahler manifold, then the Lefschetz decomposition theorem says that the primitive decomposition of forms \eqref{eq-prim-dec-forms-intro} descends to de Rham cohomology, i.e.,
\[
H^k_{dR}X=\bigoplus_{r\geq\max(k-n,0)}L^r\big(\ker\big(\Lambda:H^{k-2r}_{dR}X\to H^{k-2r-2}_{dR}X\big)\big).
\]

Cirici and Wilson recently proved a generalized Lefschetz decomposition theorem for compact almost K\"ahler manifolds. Denote by $\Delta_d:=dd^*+d^*d$ the Hodge Laplacian, where $d$ is the exterior differential and $d^*:=-*d*$ is its formal adjoint. The space of harmonic $(p,q)$-forms $\ker\Delta_d\cap A^{p,q}$ will be indicated by $\H^{p,q}_d$. They showed, see \cite[Corollary 5.4]{cirici-wilson-2}, that if $(X,J,\omega)$ is a compact almost K\"ahler manifold, then
\begin{equation}\label{eq-decomp-cirici}
\H^{p,q}_d=\bigoplus_{r\geq\max(p+q-n,0)}L^r(\H^{p-r,q-r}_d\cap P^{p-r,q-r}).
\end{equation}

Let $(X,J,\omega)$ be an almost Hermitian manifold, then other natural spaces of harmonic forms can be introduced.
The exterior differential decomposes into $d=\mu+\del+\delbar+\c\mu$, and we set $\del^*=-*\delbar*$, $\delbar^*=-*\del*$ as the formal adjoints of $\del$, $\delbar$, where $*$ is the $\C$-linear extension of the real Hodge $*$ operator.
Recall that
\begin{equation*}
\Delta_{\del}=\del\del^*+\del^*\del,\ \ \ \Delta_{\delbar}=\delbar\delbar^*+\delbar^*\delbar,
\end{equation*}
are respectively the $\del$ and $\delbar$, or Dolbeault, Laplacians, and
\begin{equation*}
\Delta_{BC}=
\del\delbar\delbar^*\del^*+
\delbar^*\del^*\del\delbar+\del^*\delbar\delbar^*\del+\delbar^*\del\del^*\delbar
+\del^*\del+\delbar^*\delbar,
\end{equation*}
and
\begin{equation*}
\Delta_{A}= \del\delbar\delbar^*\del^*+
\delbar^*\del^*\del\delbar+
\del\delbar^*\delbar\del^*+\delbar\del^*\del\delbar^*+
\del\del^*+\delbar\delbar^*,
\end{equation*}
are respectively the Bott-Chern and the Aeppli Laplacians. Denote by
\begin{equation*}
\H^{p,q}_{\del},\ \ \ \H^{p,q}_{\delbar},\ \ \  \H^{p,q}_{BC},\ \ \ \H^{p,q}_{A},
\end{equation*}
the kernels of these Laplacians intersected with the space of $(p,q)$-forms. If $X$ is compact these spaces are finite-dimensional but they do not have a cohomological counterpart. In fact, the almost complex Dolbeault, Bott-Chern and Aeppli cohomology groups might be infinite dimensional (see \cite{cirici-wilson-1}, \cite{coelho-placini-stelzig}).

In the integrable case, i.e., when $d=\del+\delbar$ and $(X,J)$ is a complex manifold, if we further assume that $X$ is compact and we endow $(X,J)$ with any Hermitian metric $\omega$, all these spaces of harmonic $(p,q)$-forms have a cohomological meaning and they coincide, i.e., 
\begin{equation*}
\H^{p,q}_d=\H^{p,q}_{\del}=\H^{p,q}_{\delbar}=\H^{p,q}_{BC}=\H^{p,q}_{A}.
\end{equation*}
Considered that, we are interested in understanding whether the primitive decomposition of harmonic forms \eqref{eq-decomp-cirici} holds also for these spaces of harmonic forms, and how these spaces are related on a given compact almost K\"ahler manifold of real dimension $2n$.

Another motivation for this problem is the following. In \cite{HZ} and \cite{HZ2} Holt and Zhang studied Dolbeault harmonic forms on the Kodaira-Thurston manifold to answer a famous question of Kodaira and Spencer which appeared as Problem 20 in Hirzebruch’s 1954 problem list \cite{hirzebruch}. Introducing an effective method to solve the PDE system associated to Dolbeault harmonic forms on the Kodaira-Thurston manifold, they proved that the dimension of the space of Dolbeault harmonic forms depends on the choice of the almost Hermitian metric. In \cite{TT}, Tomassini and the second author answered again to the same question, with a different approach, analyzing locally conformally almost K\"ahler metrics on almost complex $4$-manifolds. In \cite{PT4}, Tomassini and the first author introduced Bott-Chern and Aeppli harmonic forms on almost Hermitian manifolds and studied their relation with Dolbeault harmonic forms. 
See also \cite{Ho} and \cite{PT5} for recent results concerning the dimension of the spaces of Dolbeault and Bott-Chern harmonic $(1,1)$-forms on compact almost Hermitian $4$-manifolds.
In particular, in \cite[Proposition 6.1]{HZ}, in \cite[Theorem 3.6]{TT} and in \cite[Corollary 4.4]{PT4}, the primitive decomposition of $(1,1)$-forms  is used to deduce, in fact, primitive decompositions of Dolbeault and Bott-Chern harmonic $(1,1)$-forms on a compact almost Hermitian $4$-manifold. Given that, the study of primitive decompositions of Dolbeault, Bott-Chern, Aeppli harmonic forms can be seen as a generalisation of the just mentioned results in higher dimension $2n\ge4$ and for every bidegree $(p,q)$.

The case of primitive decompositions of Dolbeault harmonic forms is studied in \cite{cattaneo-tardini-tomassini}. In this paper, we are interested in studying primitive decompositions of Bott-Chern and Aeppli harmonic forms on a given $2n$-dimensional compact almost K\"ahler manifold $(X,J,\omega)$.  Indeed, we prove
\begin{theorem}[Theorems \ref{thm:decomp-bc}, \ref{thm:decomp-ae}, \ref{thm:decomp-bc-2}]\label{thm-main}
Let $(X,J,\omega)$ be a compact almost K\"ahler manifold of dimension $2n$. Then,
\begin{align*}
\mathcal{H}^{1,1}_{BC}&=\C\,\omega\oplus \left(\mathcal{H}^{1,1}_{BC}\cap P^{1,1}\right)\,,\\
\mathcal{H}^{1,1}_{A}&=\C\,\omega\oplus \left(\mathcal{H}^{1,1}_{A}\cap P^{1,1}\right)\,,\\
\mathcal{H}^{n-1,n-1}_{BC}&=\C\,\omega^{n-1}
\oplus L^{n-2}\left(\mathcal{H}^{1,1}_{A}\cap P^{1,1}\right)\,,\\
\mathcal{H}^{n-1,n-1}_{A}&=\C\,\omega^{n-1}
\oplus L^{n-2}\left(\mathcal{H}^{1,1}_{BC}\cap P^{1,1}\right)\,.
\end{align*}
\end{theorem}
This, in particular, generalizes the decomposition of $\mathcal{H}^{1,1}_{BC}$ in real dimension $4$ of \cite[Corollary 4.4]{PT4} to higher dimensions.
As a corollary of Theorem \ref{thm-main}, \cite[Proposition 6.1]{HZ} and \cite[Corollary 4.4]{PT4}, we derive
\begin{cor}[Corollaries \ref{cor-11}, \ref{cor-11-harmonic}]
Let $(X,J,\omega)$ be a compact almost K\"ahler manifold of dimension $4$. Then,
\[
\mathcal{H}^{1,1}_{d}=\mathcal{H}^{1,1}_{\del}=\mathcal{H}^{1,1}_{\delbar}=\mathcal{H}^{1,1}_{BC}=\mathcal{H}^{1,1}_{A}.
\]
\end{cor}
Considered that the spaces of primitive Bott-Chern and Aeppli harmonic $(1,1)$-forms turned out to be useful for the above decompositions, and that the same holds for $\del$ and Dolbeault harmonic $(1,1)$-forms by the results of \cite{cattaneo-tardini-tomassini}, we study the relations among all these spaces of primitive harmonic forms on compact almost K\"ahler manifolds. In particular, we analyze inclusions and non inclusions between these spaces. We prove
\begin{prop}[Propositions \ref{prop:bc-in-delbar-in-aeppli} and \ref{prop:equalities-deg-n}]
Let $(X,J,\omega)$ be a compact almost K\"ahler manifold of dimension $2n$.
Then, for $p+q\leq n$,
\begin{gather*}
\mathcal{H}^{p,q}_{BC}\cap P^{p,q}=
\mathcal{H}^{p,q}_{\delbar}\cap \mathcal{H}^{p,q}_{\del}\cap P^{p,q},\\
\mathcal{H}^{p,q}_{\delbar}\cap P^{p,q}\subseteq 
\mathcal{H}^{p,q}_{A}\cap P^{p,q}.
\end{gather*}
Moreover, for $p+q= n$,
\[
\mathcal{H}^{p,q}_{BC}\cap P^{p,q}=
\mathcal{H}^{p,q}_{\delbar}\cap P^{p,q}=
 \mathcal{H}^{p,q}_{\del}\cap P^{p,q}=
\mathcal{H}^{p,q}_{A}\cap P^{p,q}.
\]
\end{prop}
In Proposition \ref{prop:not-in} we show that such inclusions are in general strict, and provide other non inclusions.

Finally, we show that the primitive decompositions of Bott-Chern and Aeppli harmonic forms obtained for the bidegrees $(1,1)$ and $(n-1,n-1)$ are exclusive for these bidegrees. In fact, working on an explicit almost K\"ahler structure on the Iwasawa manifold, we show that the natural primitive decomposition of Bott-Chern harmonic forms one would expect on $(2,1)$-forms in real dimension $6$ does not hold.

The paper is organised in the following way. In section \ref{section:preliminaries} we introduce some preliminaries of almost Hermitian geometry, including some observations on the other possible definitions of the Bott-Chern and Aeppli Laplacians. In section \ref{sec-prim-dec}, we write down some trivial decompositions of Bott-Chern and Aeppli harmonic forms for the special bidegrees $(p,0)$, $(0,q)$, $(n,n-p)$ and $(n-q,n)$, and then we prove the non trivial decompositions for the bidegrees $(1,1)$ and $(n-1,n-1)$ stated in Theorem \ref{thm-main}. In section \ref{sec-relations} we study the possible inclusions and non inclusions among the spaces of primitive $\del$, $\delbar$, Bott-Chern and Aeppli harmonic forms. Finally, in section \ref{sec-dim-6}, we analyze primitive decompositions of Bott-Chern and Aeppli harmonic forms in dimension $6$.

\medskip
\noindent{\sl Acknowledgments.}  The authors would like to thank Andrea Cattaneo and Adriano Tomassini for several discussions on the subject.

\section{Preliminaries of almost Hermitian geometry}\label{section:preliminaries}

Throughout this paper, we will only consider connected manifolds without boundary.
Let $(X,J)$ be an almost complex manifold of dimension $2n$, i.e., a $2n$-differentiable manifold together with an almost complex structure $J$, that is $J\in\End(TX)$ and $J^2=-\id$. The complexified tangent bundle $T_{\C}X=TX\otimes\C$ decomposes into the two eigenspaces of $J$ associated to the eigenvalues $i,-i$, which we denote respectively by $T^{1,0}X$ and $T^{0,1}X$, giving
\begin{equation*}
T_{\C}X=T^{1,0}X\oplus T^{0,1}X.
\end{equation*}
Denoting by $\Lambda^{1,0}X$ and $\Lambda^{0,1}X$ the dual vector bundles of $T^{1,0}X$ and $T^{0,1}X$, respectively, we set
\begin{equation*}
\Lambda^{p,q}X=\bigwedge^p\Lambda^{1,0}X\wedge\bigwedge^q\Lambda^{0,1}X
\end{equation*}
to be the vector bundle of $(p,q)$-forms, and let $A^{p,q}=\Gamma(X,\Lambda^{p,q}X)$ be the space of smooth sections of $\Lambda^{p,q}X$. We denote by $A^k=\Gamma(X,\Lambda^{k}X)$ the space of $k$-forms. Note that $\Lambda^{k}X\otimes\C=\bigoplus_{p+q=k}\Lambda^{p,q}X$.

Let $f\in\cinf(X,\C)$ be a smooth function on $X$ with complex values. Its differential $df$ is contained in $A^1\otimes\C=A^{1,0}\oplus A^{0,1}$. On complex 1-forms, the exterior differential acts as
\[
d:A^1\otimes\C\to A^2\otimes\C=A^{2,0}\oplus A^{1,1}\oplus A^{0,2}.
\]
 Therefore, it turns out that the differential operates on $(p,q)$-forms as
\begin{equation*}
d:A^{p,q}\to A^{p+2,q-1}\oplus A^{p+1,q}\oplus A^{p,q+1}\oplus A^{p-1,q+2},
\end{equation*}
where we denote the four components of $d$ by
\begin{equation*}
d=\mu+\del+\delbar+\c\mu.
\end{equation*}
From the relation $d^2=0$, we derive
\begin{equation*}
\begin{cases}
\mu^2=0,\\
\mu\del+\del\mu=0,\\
\del^2+\mu\delbar+\delbar\mu=0,\\
\del\delbar+\delbar\del+\mu\c\mu+\c\mu\mu=0,\\
\delbar^2+\c\mu\del+\del\c\mu=0,\\
\c\mu\delbar+\delbar\c\mu=0,\\
\c\mu^2=0.
\end{cases}
\end{equation*}

Let $(X,J)$ be an almost complex manifold. If the almost complex structure $J$ is induced from a complex manifold structure on $X$, then $J$ is called integrable. It is equivalent to the decomposition of the exterior differential as $d=\del+\delbar$.

A Riemannian metric on $X$ for which $J$ is an isometry is called almost Hermitian.
Let $g$ be an almost Hermitian metric, the $2$-form $\omega$ such that
\begin{equation*}
\omega(u,v)=g(Ju,v)\ \ \forall u,v\in\Gamma(TX)
\end{equation*}
is called the fundamental form of $g$. We will call $(X,J,\omega)$ an almost Hermitian manifold.
We denote by $h$ the Hermitian extension of $g$ on the complexified tangent bundle $T_\C X$, and by the same symbol $g$ the $\C$-bilinear symmetric extension of $g$ on $T_\C X$. Also denote by the same symbol $\omega$ the $\C$-bilinear extension of the fundamental form $\omega$ of $g$ on $T_\C X$. 
Thanks to the elementary properties of the two extensions $h$ and $g$, we may want to consider $h$ as a Hermitian operator
$T^{1,0}X\times T^{1,0}X\to\C$ and $g$ as a $\C$-bilinear operator $T^{1,0}X\times T^{0,1}X\to\C$.
Recall that $h(u,v)=g(u,\bar{v})$ for all $u,v\in \Gamma(T^{1,0}X)$.

Let $(X,J,\omega)$ be an almost Hermitian manifold of real dimension $2n$. Extend $h$ on $(p,q)$-forms and denote the Hermitian inner product by $\langle\cdot,\cdot\rangle$.
Let $*:A^{p,q}\lr A^{n-q,n-p}$ the $\C$-linear extension of the standard Hodge $*$ operator on Riemannian manifolds with respect to the volume form $\vol=\frac{\omega^n}{n!}$, i.e., $*$ is defined by the relation
\[
\alpha\wedge{*\c\beta}=\langle\alpha,\beta\rangle\vol\ \ \ \forall\alpha,\beta\in A^{p,q}.
\]
Then the operators
\begin{equation*}
d^*=-*d*,\ \ \ \mu^*=-*\c\mu*,\ \ \ \del^*=-*\delbar*,\ \ \ \delbar^*=-*\del*,\ \ \ \c\mu^*=-*\mu*,
\end{equation*}
are the formal adjoint operators respectively of $d,\mu,\del,\delbar,\c\mu$. Recall that $\Delta_{d}=dd^*+d^*d$ is the Hodge Laplacian, and, as in the integrable case, set 
\begin{equation*}
\Delta_{\del}=\del\del^*+\del^*\del,\ \ \ \Delta_{\delbar}=\delbar\delbar^*+\delbar^*\delbar,
\end{equation*}
respectively as the $\del$ and $\delbar$ Laplacians. Again, as in the integrable case, set
\begin{equation*}
\Delta_{BC}=
\del\delbar\delbar^*\del^*+
\delbar^*\del^*\del\delbar+\del^*\delbar\delbar^*\del+\delbar^*\del\del^*\delbar
+\del^*\del+\delbar^*\delbar,
\end{equation*}
and
\begin{equation*}
\Delta_{A}= \del\delbar\delbar^*\del^*+
\delbar^*\del^*\del\delbar+
\del\delbar^*\delbar\del^*+\delbar\del^*\del\delbar^*+
\del\del^*+\delbar\delbar^*,
\end{equation*}
respectively as the Bott-Chern and the Aeppli Laplacians. Note that
\begin{equation}\label{bc-a-duality}
*\Delta_{BC}=\Delta_{A}*\ \ \ \Delta_{BC}*=*\Delta_{A}.
\end{equation}

If $X$ is compact, then we easily deduce the following relations
\begin{equation*}
\begin{cases}
\Delta_{d}=0\ &\iff\ d=0,\ d*=0,\\
\Delta_{\del}=0\ &\iff\ \del=0,\ \delbar*=0,\\
\Delta_{\delbar}=0\ &\iff\ \delbar=0,\ \del*=0,\\
\Delta_{BC}=0\ &\iff \del=0,\ \delbar=0,\ \del\delbar*=0,\\
\Delta_{A}=0\ &\iff \del*=0,\ \delbar*=0,\ \del\delbar=0,
\end{cases}
\end{equation*}
which characterize the spaces of harmonic forms
\begin{equation*}
\H^{k}_{d},\ \ \ \H^{p,q}_{\del},\ \ \ \H^{p,q}_{\delbar},\ \ \  \H^{p,q}_{BC},\ \ \ \H^{p,q}_{A},
\end{equation*}
defined as the spaces of forms which are in the kernel of the associated Laplacians.
All these Laplacians are elliptic operators on the almost Hermitian manifold $(X,J,\omega)$ (cf. \cite{hirzebruch}, \cite{PT4}), implying that all the spaces of harmonic forms are finite dimensional when the manifold is compact. Denote by
\begin{equation*}
b^{k},\ \ \ h^{p,q}_{\del},\ \ \ h^{p,q}_{\delbar},\ \ \ h^{p,q}_{BC},\ \ \ h^{p,q}_{A}
\end{equation*}
respectively the real dimension of $\H^k_d$ and the complex dimensions of $\H^{p,q}_{\del}$, $\H^{p,q}_{\delbar}$, $\H^{p,q}_{BC}$, $\H^{p,q}_{A}$.

\begin{rem}
By \eqref{bc-a-duality}, note that 
\begin{equation}\label{eq-star-bc-ae}
*\H^{p,q}_{BC}=\H^{n-q,n-p}_{A},\ \ \ *\H^{p,q}_{A}=\H^{n-q,n-p}_{BC}.
\end{equation}
In the following, it will be often useful to study the spaces $\H^{p,q}_{BC}$ and $\H^{p,q}_{A}$ for $p+q\le n$ in order to obtain information respectively on the spaces $\H^{n-q,n-p}_{A}$ and $\H^{n-q,n-p}_{BC}$.
\end{rem}

\begin{rem}
We observe that, since the operators $\del$ and $\delbar$ do not anticommute in the non integrable setting, we could have made a different choice for the Bott-Chern and Aeppli Laplacians, namely we could have taken
$$
\Delta_{BC,2}=\delbar\del\del^*\delbar^*+
\del^*\delbar^*\delbar\del+\delbar^*\del\del^*\delbar+\del^*\delbar\delbar^*\del
+\delbar^*\delbar+\del^*\del
$$
$$
\Delta_{A,2}=\delbar\del\del^*\delbar^*+
\del^*\delbar^*\delbar\del+
\delbar\del^*\del\delbar^*+\del\delbar^*\delbar\del^*+
\delbar\delbar^*+\del\del^*
$$
However, we notice that they differ by conjugation, namely
$$
\overline\Delta_{BC}=\Delta_{BC,2}\,,\qquad
\overline\Delta_{A}=\Delta_{A,2}.
$$
Hence,
$$
\alpha\in\Ker\Delta_{BC}\qquad\iff\qquad
\overline\alpha\in\Ker\Delta_{BC,2},
$$
therefore it is not restrictive to study only the space $\mathcal{H}^{\bullet,\bullet}_{BC}(X)$. A similar argument shows that it is not restrictive to study only the space $\mathcal{H}^{\bullet,\bullet}_{A}(X)$. 
\end{rem}

\section{Primitive decompositions of Bott-Chern harmonic forms}\label{sec-prim-dec}

In the following we are going to show that, in special bidegrees, we have natural primitive decompositions for Bott-Chern harmonic forms on compact almost K\"ahler manifolds. We first need to introduce some notations and recall some well known facts about primitive forms.
Let $(X,J,\omega)$ be a $2n$-dimensional almost Hermitian manifold. 
We denote with
$$
L:\Lambda^kX\to\Lambda^{k+2}X\,,\quad \alpha\mapsto\omega\wedge\alpha
$$
the Lefschetz operator and with
$$
\Lambda:\Lambda^kX\to\Lambda^{k-2}X\,,\quad \Lambda=-*L*
$$
its dual.
A differential $k$-form $\alpha_k$ on $X$, for $k\leq n$, is said to be {\em primitive} if $\Lambda\alpha_k=0$, or equivalently $L^{n-k+1}\alpha_k=0$. Then, the following vector bundle decomposition holds (see e.g., \cite[p. 26, Th\'eor\`eme 3]{weil})
\begin{equation}\label{eq-prim-dec-forms}
\Lambda^kX=\bigoplus_{r\geq\max(k-n,0)}L^r(P^{k-2r}X),
\end{equation}
where we denoted
$$
P^{s}X:=\ker\big(\Lambda:\Lambda^{s}X\to\Lambda^{s-2}X\big)
$$
the bundle of primitive $s$-forms. Accordingly to such decomposition, given any $k$-form $\alpha_k\in\Lambda^kX$, we can write 
\begin{equation}\label{primitive-bundle-decomposition}
\alpha_k=\sum_{r\geq\max(k-n,0)}\frac{1}{r!}L^r\beta_{k-2r},
\end{equation}
where $\beta_{k-2r}\in P^{k-2r}X$, namely
$$\Lambda\beta_{k-2r} =0,
$$ 
or equivalently 
$$
L^{n-k+2r+1}\beta_{k-2r}=0.
$$
Furthermore, the decomposition above is compatible with the bidegree decomposition on the bundle of complex $k$-forms $\Lambda_\C^kX$ induced by $J$, that is 
$$
P_\C^kX=\bigoplus_{p+q=k}P^{p,q}X,
$$
where 
$$
P^{p,q}X=P^k_\C X\cap\Lambda^{p,q}X.
$$
For any given $\beta_k\in P^kX$, we have the following formula (cf. \cite[p. 23, Th\'eor\`eme 2]{weil})
\begin{equation}\label{*-primitive}
*L^r\beta_k=(-1)^{\frac{k(k+1)}{2}}\frac{r!}{(n-k-r)!}L^{n-k-r}J\beta_k.
\end{equation}

Let us set $P^s:=\Gamma(X,P^sX)$ and $P^{p,q}:=\Gamma(X,P^{p,q}X)$.
We recall that the map $L^h:\Lambda^kX\to\Lambda^{k+2h}X$ is injective for $h+k\le n$ and is surjective for $h+k\ge n$.

\begin{rem}
In \cite[Corollary 5.4]{cirici-wilson-2} it is proven that on a $2n$-dimensional compact almost K\"ahler manifold such primitive decompositions pass to $d$-harmonic forms, namely
$$
\mathcal{H}_d^{p,q}=\bigoplus_{r\geq\max(p+q-n,0)}L^r(
\mathcal{H}_d^{p-r,q-r}\cap P^{p-r,q-r}).
$$
In fact, this holds true also for the spaces of harmonic forms introduced in \cite{TT0}. More precisely, setting
$$
\bar\delta:=\delbar+\mu,\qquad
\delta:=\del+\bar\mu
$$
one has, on compact almost K\"ahler manifolds, for every $p,q$ (\cite[Proposition 6.2, Theorem 6.7]{TT0})
$$
\mathcal{H}_d^{p,q}=\mathcal{H}_{\bar\delta}^{p,q}=
\mathcal{H}_{\delta}^{p,q}\,.
$$
\end{rem}

Here we are interested in investigating when the decomposition \eqref{eq-prim-dec-forms} descends to Bott-Chern and Aeppli harmonic forms. Note that, since $(p,0)$-forms and $(0,q)$-forms are trivially primitive, we immediately derive for $p,q\leq n$
\begin{align*}
\H^{p,0}_{BC}&=\H^{p,0}_{BC}\cap P^{p,0},\ \ \ &\H^{0,q}_{BC}&=\H^{0,q}_{BC}\cap P^{0,q},\\
\H^{p,0}_{A}&=\H^{p,0}_{A}\cap P^{p,0},\ \ \ &\H^{0,q}_{A}&=\H^{0,q}_{A}\cap P^{0,q}.
\end{align*}
Applying the Hodge $*$ operator to the previous trivial primitive decompositions of the spaces of Bott-Chern and Aeppli harmonic forms, from \eqref{eq-star-bc-ae} and \eqref{*-primitive} we easily obtain respectively
\begin{align*}
\H^{n,n-p}_{A}&=L^{n-p}\left(\H^{p,0}_{BC}\cap P^{p,0}\right),\ \ \ &\H^{n-q,n}_{A}&=L^{n-q}\left(\H^{0,q}_{BC}\cap P^{0,q}\right),\\
\H^{n,n-p}_{BC}&=L^{n-p}\left(\H^{p,0}_{A}\cap P^{p,0}\right),\ \ \ &\H^{n-q,n}_{BC}&=L^{n-q}\left(\H^{0,q}_{A}\cap P^{0,q}\right).
\end{align*}
In particular, taking $p=q=n$ we obtain
$$
\H^{n,0}_{BC}=\H^{n,0}_{A}\qquad\text{and}\qquad
\H^{0,n}_{BC}=\H^{0,n}_{A}\,.
$$
In fact, this can be obtained directly using formula (\ref{*-primitive}) as done in Proposition \ref{prop:equalities-deg-n} recalling that $(n,0)$-forms and $(0,n)$-forms are trivially primitive. 
We find more interesting primitive decompositions when we look at the space of Bott-Chern and Aeppli harmonic forms of bidegree $(1,1)$.

\begin{theorem}\label{thm:decomp-bc}
Let $(X,J,\omega)$ be a compact almost K\"ahler manifold of dimension $2n$. Then,
\[
\mathcal{H}^{1,1}_{BC}=\C\,\omega\oplus \left(\mathcal{H}^{1,1}_{BC}\cap P^{1,1}\right)\,.
\]
\end{theorem}
\begin{proof}
Let $\psi\in\H^{1,1}_{BC}$, i.e., $\psi\in A^{1,1}$ and
\begin{equation}\label{charact-psi}
\del\psi=0,\ \delbar\psi=0,\ \del\delbar*\psi=0.
\end{equation}
By \eqref{primitive-bundle-decomposition}, we derive
\[
\psi=f\omega+\gamma,
\]
where $f$ is a smooth function with complex values on $X$, and $\gamma$ is a primitive $(1,1)$-form, i.e., $\Lambda\gamma=0$. Since both $f$ and $\gamma$ are primitive forms, we apply \eqref{*-primitive} to compute $*\psi$. We obtain
\[
*\psi=\frac{\omega^{n-1}}{(n-1)!}f-\frac{\omega^{n-2}}{(n-2)!}\wedge\gamma.
\]
Now, from \eqref{charact-psi} and from the assumption that the metric is almost K\"ahler, it follows that
\begin{align*}
0&=\del\psi=\del f\wedge\omega+\del\gamma,\\
0&=\delbar\psi=\delbar f\wedge \omega+\delbar\gamma,\\
0&=\frac{\omega^{n-1}}{(n-1)!}\wedge\del\delbar f-\frac{\omega^{n-2}}{(n-2)!}\wedge\del\delbar\gamma\\
&=\frac{\omega^{n-1}}{(n-1)!}\wedge\del\delbar f-\frac{\omega^{n-2}}{(n-2)!}\wedge\del(-\delbar f\wedge \omega)\\
&=\frac{\omega^{n-1}}{(n-1)!}\wedge\del\delbar f+\frac{\omega^{n-1}}{(n-2)!}\wedge\del\delbar f\\
&=\left(\frac1{(n-1)!}+\frac1{(n-2)!}\right)\omega^{n-1}\wedge\del\delbar f.
\end{align*}
Arguing like in \cite[Theorem 4.3]{PT4} or in \cite[Proposition 3.4]{TT}, one can show that the differential operator $L:\cinf(X,\C)\to\cinf(X,\C)$ defined by
\[
L:f\mapsto -i*(\del\delbar f\wedge \omega^{n-1})
\]
is strongly elliptic and, being $f\in\Ker L$, it follows that $f$ is a complex constant by the maximum principle.
Since $f\in\C$, the equations in \eqref{charact-psi} are equivalent to
\[
\del\gamma=0,\ \delbar\gamma=0.
\]
Note that $\del\gamma=\delbar\gamma=0$ and $\Lambda\gamma=0$ implies $\del\delbar*\gamma=0$.
Summing up, we showed that if $\psi=f\omega+\gamma\in\H^{1,1}_{BC}$, where $f\in\cinf(X,\C)$ and $\gamma\in P^{1,1}$, then $f\in\C$ and $\gamma\in\H^{1,1}_{BC}$ proving the inclusion $\subseteq$ of the statement. The converse inclusion $\supseteq$ is trivial, therefore the theorem is proved.
\end{proof}

For Aeppli harmonic $(1,1)$-forms, we find the following similar decomposition.

\begin{theorem}\label{thm:decomp-ae}
Let $(X,J,\omega)$ be a compact almost K\"ahler manifold of dimension $2n$. Then,
\[
\mathcal{H}^{1,1}_{A}=\C\,\omega\oplus \left(\mathcal{H}^{1,1}_{A}\cap P^{1,1}\right)\,.
\]
\end{theorem}
\begin{proof}
Let $\psi\in\H^{1,1}_{A}$, i.e., $\psi\in A^{1,1}$ and
\begin{equation}\label{charact-psi-ae}
\del*\psi=0,\ \delbar*\psi=0,\ \del\delbar\psi=0.
\end{equation}
By \eqref{primitive-bundle-decomposition}, we derive
\[
\psi=f\omega+\gamma,
\]
where $f$ is a smooth function with complex values on $X$, and $\gamma$ is a primitive $(1,1)$-form, i.e., $\Lambda\gamma=0$. Since both $f$ and $\gamma$ are primitive forms, we apply \eqref{*-primitive} to compute $*\psi$. We obtain
\[
*\psi=\frac{\omega^{n-1}}{(n-1)!}f-\frac{\omega^{n-2}}{(n-2)!}\wedge\gamma.
\]
Now, from \eqref{charact-psi} and from the assumption that the metric is almost K\"ahler, it follows that
\begin{align}
0&=\del*\psi=\frac{\omega^{n-1}}{(n-1)!}\wedge\del f-\frac{\omega^{n-2}}{(n-2)!}\wedge\del\gamma,\notag\\
0&=\delbar*\psi=\frac{\omega^{n-1}}{(n-1)!}\wedge\delbar f-\frac{\omega^{n-2}}{(n-2)!}\wedge\delbar\gamma,\label{eq2}\\
0&=\del\delbar f\wedge\omega+\del\delbar\gamma.\label{eq3}
\end{align}
We apply $L^{n-2}$ to \eqref{eq3}, obtaining
\begin{align*}
0&=\omega^{n-1}\wedge\del\delbar f+\omega^{n-2}\wedge\del\delbar\gamma.
\end{align*}
We apply $\del$ to \eqref{eq2}, deriving
\begin{align*}
0&=\frac{\omega^{n-1}}{(n-1)!}\wedge\del\delbar f-\frac{\omega^{n-2}}{(n-2)!}\wedge\del\delbar\gamma.
\end{align*}
Combining the last two equations, we find
\begin{align*}
0&=\frac{\omega^{n-1}}{(n-1)!}\wedge\del\delbar f+\frac{\omega^{n-1}}{(n-2)!}\wedge\del\delbar f\\
&=\left(\frac1{(n-1)!}+\frac1{(n-2)!}\right)\omega^{n-1}\wedge\del\delbar f.
\end{align*}
Arguing like in Theorem \ref{thm:decomp-bc}, it follows that $f$ is a complex constant.
Since $f\in\C$, the equations in \eqref{charact-psi-ae} are equivalent to
\[
\del*\gamma=0,\ \delbar*\gamma=0,\ \del\delbar\gamma=0.
\]
Summing up, we showed that if $\psi=f\omega+\gamma\in\H^{1,1}_{A}$, where $f\in\cinf(X,\C)$ and $\gamma\in P^{1,1}$, then $f\in\C$ and $\gamma\in\H^{1,1}_{A}$ proving the inclusion $\subseteq$ of the statement. The converse inclusion $\supseteq$ is trivial, therefore the theorem is proved.
\end{proof}

As a corollary of the previous results we obtain also the following decompositions of the spaces of $(n-1,n-1)$ Bott-Chern and Aeppli harmonic forms.
\begin{theorem}\label{thm:decomp-bc-2}
Let $(X,J,\omega)$ be a compact almost K\"ahler manifold of dimension $2n$. Then, the following decompositions hold
\begin{equation}\label{dec-bc-n-1}
\mathcal{H}^{n-1,n-1}_{BC}=\C\,\omega^{n-1}
\oplus L^{n-2}\left(\mathcal{H}^{1,1}_{A}\cap P^{1,1}\right)\,,
\end{equation}
\begin{equation}\label{dec-ae-n-1}
\mathcal{H}^{n-1,n-1}_{A}=\C\,\omega^{n-1}
\oplus L^{n-2}\left(\mathcal{H}^{1,1}_{BC}\cap P^{1,1}\right)\,.
\end{equation}
\end{theorem}
\begin{proof}
Decompositions \eqref{dec-bc-n-1} and \eqref{dec-ae-n-1} follow respectively from the decompositions of Theorems \ref{thm:decomp-ae} and \ref{thm:decomp-bc} applying the Hodge star operator and using formulae \eqref{*-primitive} and \eqref{eq-star-bc-ae}.
\end{proof}

From Theorems \ref{thm:decomp-bc}, \ref{thm:decomp-ae} and \ref{thm:decomp-bc-2}, it immediately follows

\begin{cor}\label{cor-11}
Let $(X,J,\omega)$ be a compact almost K\"ahler manifold of dimension $4$. Then,
\[
\mathcal{H}^{1,1}_{BC}=\mathcal{H}^{1,1}_{A}=\C\,\omega
\oplus \left(\mathcal{H}^{1,1}_{BC}\cap P^{1,1}\right).
\]
\end{cor}

Combining Corollary \ref{cor-11} with \cite[Proposition 6.1]{HZ} and \cite[Corollary 4.4]{PT4}, we derive

\begin{cor}\label{cor-11-harmonic}
Let $(X,J,\omega)$ be a compact almost K\"ahler manifold of dimension $4$. Then,
\[
\mathcal{H}^{1,1}_{d}=\mathcal{H}^{1,1}_{\del}=\mathcal{H}^{1,1}_{\delbar}=\mathcal{H}^{1,1}_{BC}=\mathcal{H}^{1,1}_{A}.
\]
\end{cor}

These last two results will be generalized in Proposition \ref{prop:equalities-deg-n}.

\begin{rem}
Notice that, in fact, on $4$-dimensional almost Hermitian manifolds the primitive decomposition of $\mathcal{H}^{1,1}_{BC}$ was proved in \cite[Corollary 4.4]{PT4}, where it is shown that
\[
\mathcal{H}^{1,1}_{BC}\cap P^{1,1}=\{\alpha\in A^{1,1}\,:\,\Delta_d\alpha=0,\ *\alpha=-\alpha\}.
\]
\end{rem}

\section{Relations among the spaces of primitive harmonic forms}\label{sec-relations}

We saw that the spaces of primitive Bott-Chern and Aeppli harmonic forms are important in Theorems \ref{thm:decomp-bc}, \ref{thm:decomp-ae} and \ref{thm:decomp-bc-2}. Moreover, the spaces of primitive $\del$- and Dolbeault harmonic forms play a similar role as shown in \cite{cattaneo-tardini-tomassini}. Let us then study the possible inclusions and non inclusions between these spaces.

\begin{prop}\label{prop:bc-in-delbar-in-aeppli}
Let $(X,J,\omega)$ be a compact almost K\"ahler manifold of dimension $2n$.
Then, for $p+q\leq n$,
\begin{gather}\label{eq-bc-del-delbar}
\mathcal{H}^{p,q}_{BC}\cap P^{p,q}=
\mathcal{H}^{p,q}_{\delbar}\cap \mathcal{H}^{p,q}_{\del}\cap P^{p,q},\\
\mathcal{H}^{p,q}_{\delbar}\cap P^{p,q}\subseteq 
\mathcal{H}^{p,q}_{A}\cap P^{p,q}.\label{eq-delbar-ae}
\end{gather}
In particular,
\begin{gather*}
\mathcal{H}^{p,q}_{BC}\cap P^{p,q}\subseteq
\mathcal{H}^{p,q}_{\delbar}\cap P^{p,q},\\
\mathcal{H}^{p,q}_{BC}\cap P^{p,q}\subseteq
\mathcal{H}^{p,q}_{\del}\cap P^{p,q},\\
\mathcal{H}^{p,q}_{BC}\cap P^{p,q}\subseteq 
\mathcal{H}^{p,q}_{A}\cap P^{p,q}.
\end{gather*}
\end{prop}
\begin{proof}
We start by showing \eqref{eq-bc-del-delbar}.
Let $\alpha\in \mathcal{H}^{p,q}_{\delbar}\cap \mathcal{H}^{p,q}_{\del}\cap P^{p,q}$, i.e.,
$$
\del\alpha=\delbar\alpha=\del*\alpha=\delbar*\alpha=0.
$$
Hence, $\del\alpha=\delbar\alpha=\del\delbar*\alpha=0$ and so $\alpha\in \mathcal{H}^{p,q}_{BC}$. Notice that in fact we did not use that $\alpha$ is a primitive form.\\
Viceversa, let $\alpha\in \mathcal{H}^{p,q}_{BC}\cap P^{p,q}$, i.e.,
$$
\del\alpha=\delbar\alpha=\del\delbar*\alpha=0.
$$
Since $\alpha$ is primitive, $*\alpha=c\,\omega^{n-p-q}\wedge\alpha$, with $c=\frac{(-1)^{\frac{(p+q)(p+q+1)}{2}}}{(n-p-q)!}i^{p-q}$ and since $\omega$ is $d$-closed we have
$$
\del*\alpha=c\,\del(\omega^{n-p-q}\wedge\alpha)=
c\,\omega^{n-p-q}\wedge\del\alpha=0
$$
and similarly
$$
\delbar*\alpha=c\,\delbar(\omega^{n-p-q}\wedge\alpha)=
c\,\omega^{n-p-q}\wedge\delbar\alpha=0
$$
so $\alpha\in \mathcal{H}^{p,q}_{\delbar}\cap \mathcal{H}^{p,q}_{\del}\cap P^{p,q}$.

Now we show \eqref{eq-delbar-ae}.
Let $\alpha\in \mathcal{H}^{p,q}_{\delbar}\cap P^{p,q}$, i.e.,
$$
\delbar\alpha=\del*\alpha=0.
$$
In particular $
\del\delbar\alpha=\del*\alpha=0$.
To prove $\delbar*\alpha=0$, which implies $\alpha\in \mathcal{H}^{p,q}_{A}\cap P^{p,q}$, recall that $*\alpha=c\,\omega^{n-p-q}\wedge\alpha$, and again since $\omega$ is $d$-closed we have
\begin{equation*}
\delbar*\alpha=c\,\delbar(\omega^{n-p-q}\wedge\alpha)=
c\,\omega^{n-p-q}\wedge\delbar\alpha=0.\qedhere
\end{equation*}
\end{proof}

In fact, in some special bidegrees also the other inclusions hold and so, for primitive forms, all the spaces of harmonic forms coincide.
\begin{prop}\label{prop:equalities-deg-n}
Let $(X,J,\omega)$ be a compact almost K\"ahler manifold of dimension $2n$.
Then, for $p+q= n$,
$$
\mathcal{H}^{p,q}_{BC}\cap P^{p,q}=
\mathcal{H}^{p,q}_{\delbar}\cap P^{p,q}=
 \mathcal{H}^{p,q}_{\del}\cap P^{p,q}=
\mathcal{H}^{p,q}_{A}\cap P^{p,q}.
$$
\end{prop}
\begin{proof}
Let $\alpha\in P^{p,q}$ for $p+q=n$, then by Formula (\ref{*-primitive}) $*\alpha=c_{p,q}\alpha$, with $c_{p,q}=(-1)^{\frac{n(n+1)}{2}}i^{p-q}$. Therefore,
$$
\del\alpha=0 \qquad \iff\qquad \del*\alpha=0\qquad \iff\qquad \delbar^*\alpha=0
$$
and
$$
\delbar\alpha=0 \qquad \iff\qquad \delbar*\alpha=0\qquad \iff\qquad \del^*\alpha=0\,.
$$
The equalities follow then directly from the definitions.
For instance, to prove $\mathcal{H}^{p,q}_{\delbar}\cap P^{p,q}\subseteq \mathcal{H}^{p,q}_{BC}\cap P^{p,q}$, let $\alpha\in \mathcal{H}^{p,q}_{\delbar}\cap P^{p,q}$, then $\delbar\alpha=0$ and $\del*\alpha=0$. Hence, by the previous observation on has also that $\delbar*\alpha=0$ and $\del\alpha=0$, giving in particular that $\del\delbar*\alpha=0$, $\delbar\alpha=0$ and $\del\alpha=0$ which means that $\alpha\in \mathcal{H}^{p,q}_{BC}\cap P^{p,q}$.\\
The other inclusions can be proved similarly.
\end{proof}

For $p+q<n$, the inclusions of Proposition \ref{prop:bc-in-delbar-in-aeppli} can be summed up in the following diagram.

\begin{equation*}
\begin{tikzcd}
\  & \H^{p,q}_{\delbar}\cap P^{p,q}  \arrow[rd,"\subseteq"] &\  \\
\H^{p,q}_{BC}\cap P^{p,q}  \arrow[rd,"\subseteq"] \arrow[ru,"\subseteq"] \arrow[rr,"\subseteq"]  &  \     &   \H^{p,q}_{A}\cap P^{p,q}\\
\ & \H^{p,q}_{\del}\cap P^{p,q}&\ 
\end{tikzcd}
\end{equation*}

Combining Proposition \ref{prop:bc-in-delbar-in-aeppli} for $(p,q)=(1,1)$ with Theorems \ref{thm:decomp-bc}, \ref{thm:decomp-ae} and \ref{thm:decomp-bc-2} we obtain the following

\begin{cor}
Let $(X,J,\omega)$ be a compact almost K\"ahler manifold of dimension $2n$.
Then,
$$
\mathcal{H}^{1,1}_{BC}(X)\subseteq \mathcal{H}^{1,1}_{A}(X),
$$
and
$$
\mathcal{H}^{n-1,n-1}_{A}(X)\subseteq 
\mathcal{H}^{n-1,n-1}_{BC}(X)\,.
$$
\end{cor}

Now we study whether the inclusions in Proposition \ref{prop:bc-in-delbar-in-aeppli} are strict, and in general if there are other inclusions between the spaces of primitive harmonic forms.

\begin{prop}\label{prop:not-in}
There exists a compact almost K\"ahler 6-manifold $(X,J,\omega)$ such that
\begin{align}
\mathcal{H}^{1,1}_{\delbar}\cap P^{1,1}&\not\subseteq
\mathcal{H}^{1,1}_{BC}\cap P^{1,1},\label{eq-1}\\
\mathcal{H}^{1,1}_{\delbar}\cap P^{1,1}&\not\subseteq
\mathcal{H}^{1,1}_{\del}\cap P^{1,1},\label{eq-2}\\
\mathcal{H}^{1,1}_{A}\cap P^{1,1}&\not\subseteq 
\mathcal{H}^{1,1}_{BC}\cap P^{1,1},\label{eq-2.5}\\
\mathcal{H}^{1,1}_{A}\cap P^{1,1}&\not\subseteq 
\mathcal{H}^{1,1}_{\del}\cap P^{1,1},\label{eq-3}\\
\mathcal{H}^{1,1}_{\del}\cap P^{1,1}&\not\subseteq
\mathcal{H}^{1,1}_{\delbar}\cap P^{1,1},\label{eq-4}\\
\mathcal{H}^{1,1}_{\del}\cap P^{1,1}&\not\subseteq
\mathcal{H}^{1,1}_{BC}\cap P^{1,1},\label{eq-5}\\
\mathcal{H}^{1,1}_{\del}\cap P^{1,1}&\not\subseteq 
\mathcal{H}^{1,1}_{A}\cap P^{1,1}.\label{eq-6}
\end{align}
\end{prop}
\begin{proof}
We refer to Example \ref{ex:toro} for the proof of this Proposition.
\end{proof}

\begin{ex}\label{ex:toro}
We recall the following construction of \cite{cattaneo-tardini-tomassini}.
Let $X=\mathbb{T}^6=\mathbb{Z}^6\backslash \mathbb{R}^6$ be the $6$-dimensional torus with $(x^1,x^2,x^3,y^1,y^2,y^3)$ coordinates on $\mathbb{R}^6$. Let $g=g(x^3,y^3)$ be a non-constant function on $\mathbb{T}^6$.
We define an almost complex structure $J$ setting as global co-frame of $(1,0)$-forms 
$$
\varphi^1:=e^gdx^1+i\,e^{-g}dy^1,\quad
\varphi^2:=dx^2+i\,dy^2,\quad
\varphi^3:=dx^3+i\,dy^3\,.
$$
The structure equations are
$$
d\varphi^1=V_3(g)\varphi^{3\bar1}-\bar V_3(g)\varphi^{\bar1\bar3}\,,\quad
d\varphi^2=d\varphi^3=0,
$$
where $\left\lbrace V_1,V_2,V_3\right\rbrace$ denotes the global frame of vector fields dual to $\left\lbrace\varphi^1,\varphi^2,\varphi^3\right\rbrace$. Notice that in particular $J$ is not integrable.\\
Then, the $(1,1)$-form
$$
\omega:=\frac{i}{2}\varphi^{1\bar 1}+\frac{i}{2}\varphi^{2\bar 2}+\frac{i}{2}\varphi^{3\bar 3}
$$
is a compatible symplectic structure, namely $(J,\omega)$ is an almost K\"ahler structure on $\mathbb{T}^6$.

We will show that the form $\phi^{2\bar1}$ verifies claims \eqref{eq-1}, \eqref{eq-2}, \eqref{eq-2.5}, \eqref{eq-3}, while the form $\phi^{1\bar2}$ verifies \eqref{eq-4}, \eqref{eq-5}, \eqref{eq-6}.

First, note that
\[
\delbar\varphi^{2\bar1}=0
\]
and
$$
\del*\varphi^{2\bar1}=-\omega\wedge\del\varphi^{2\bar1}=
\overline{V_3(g)}\omega\wedge\varphi^{12\bar 3}=0,
$$
thus 
$$
\varphi^{2\bar1}\in \mathcal{H}^{1,1}_{\delbar}\cap P^{1,1}\subseteq \mathcal{H}^{1,1}_{A}\cap P^{1,1}.
$$
On the other hand 
$$
\del\varphi^{2\bar1}=-\overline{V_3(g)}\varphi^{12\bar 3}\neq 0
$$
and so
$$
\varphi^{2\bar1}\notin \mathcal{H}^{1,1}_{\del}\cap P^{1,1}\supseteq\mathcal{H}^{1,1}_{BC}\cap P^{1,1}\,.
$$
This proves \eqref{eq-1}, \eqref{eq-2}, \eqref{eq-2.5}, \eqref{eq-3}.

Now, note that 
\[
\del\phi^{1\bar2}=0
\]
and
$$
\delbar*\varphi^{1\bar2}=-\omega\wedge\delbar\varphi^{1\bar2}=
-{V_3(g)}\omega\wedge\varphi^{3\bar 1\bar2}=0,
$$
thus
$$
\varphi^{1\bar2}\in \mathcal{H}^{1,1}_{\del}\cap P^{1,1}.
$$
On the other hand
\[
\del\delbar\varphi^{1\bar 2}=\del(V_3(g)\varphi^{3\bar1\bar2})=-V_3(g)\overline{V_3(g)}\varphi^{3\bar31\bar2}\ne0
\]
and so
$$
\varphi^{1\bar2}\notin \mathcal{H}^{1,1}_{A}\cap P^{1,1}\supseteq\mathcal{H}^{1,1}_{\delbar}\cap P^{1,1}\supseteq\mathcal{H}^{1,1}_{BC}\cap P^{1,1}\,.
$$
This proves \eqref{eq-4}, \eqref{eq-5}, \eqref{eq-6}.\qed
\end{ex} 

Combining Propositions \ref{prop:bc-in-delbar-in-aeppli} and \ref{prop:not-in} one finds the following diagram of strict inclusions.

\begin{equation}\label{diagram-incl}
\begin{tikzcd}
\  & \H^{p,q}_{\delbar}\cap P^{p,q}  \arrow[rd,"\subseteq"] &\  \\
\H^{p,q}_{BC}\cap P^{p,q}  \arrow[rd,"\subsetneq"] \arrow[ru,"\subsetneq"] \arrow[rr,"\subsetneq"]  &  \     &   \H^{p,q}_{A}\cap P^{p,q}\\
\ & \H^{p,q}_{\del}\cap P^{p,q}&\ 
\end{tikzcd}
\end{equation}

The remaining non-inclusions of Propositions \ref{prop:not-in} (which are not already included in diagram \eqref{diagram-incl}) can be summed up in the following diagram.

\begin{equation}\label{diagram-non-incl}
\begin{tikzcd}
\  & \H^{p,q}_{\delbar}\cap P^{p,q} \arrow[dd,"\not\subseteq", shift left=1ex]   &\  \\
\H^{p,q}_{BC}\cap P^{p,q}     &  \     &   \H^{p,q}_{A}\cap P^{p,q} \arrow[ld,"\not\subseteq", shift left=2.5ex] \\
\ & \H^{p,q}_{\del}\cap P^{p,q} \arrow[ru,"\not\subseteq", shift right=0.5ex] \arrow[uu,"\not\subseteq", shift left=1ex] &\ 
\end{tikzcd}
\end{equation}

It remains open to understand if $\H^{p,q}_{A}\cap P^{p,q}$ is either contained or not in $\H^{p,q}_{\delbar}\cap P^{p,q}$ in general.

\begin{rem}
We notice that in \cite{TT0} the spaces of Bott-Chern and Aeppli harmonic forms were introduced using the operators $\bar\delta$ and $\delta$. In fact, it was shown that with respect to such operators, on compact almost K\"ahler manifolds one has the usual equalities (that are true for K\"ahler manifolds), namely (cf. \cite[Proposition 6.2, Theorem 6.7, Proposition 6.10, Corollary 6.12]{TT0})
$$
\mathcal{H}^{p,q}_{BC(\delta,\bar\delta)}(X)=
 \mathcal{H}^{p,q}_{A(\delta,\bar\delta)}(X)=
 \mathcal{H}^{p,q}_{\bar\delta}(X)=
\mathcal{H}^{p,q}_{\delta}(X)=
\mathcal{H}^{p,q}_{d}(X).
$$
\end{rem}

\section{Primitive decompositions of harmonic forms in dimension $6$}\label{sec-dim-6}

Let $(X,J,\omega)$ be a compact almost K\"ahler manifold of dimension $2n$. In Section \ref{sec-prim-dec} we saw that the primitive decompositions of $(p,q)$-forms descend to Bott-Chern and Aeppli harmonic forms for the special bidegrees $(1,1)$, $(p,0)$ and $(0,q)$.  By Bott-Chern and Aeppli duality, we saw that we can also deduce primitive decompositions for the bidegrees $(n-1,n-1)$, $(n,n-p)$ and $(n-q,n)$. But do these decompositions hold only for these special bidegrees? Are there other bidegrees with nice primitive decomposisions of the spaces of Bott-Chern and Aeppli harmonic forms? For $2n=2,4$, since the previous bidegrees are all the possible bidegrees, the situation is well understood. Therefore, to answer our question, we should investigate what happens in the dimension $2n=6$.

If $2n=6$, then the only bidegrees for which we do not still have primitive decompositions of the spaces of Bott-Chern and Aeppli harmonic forms are $(2,1)$, $(1,2)$. Let us focus on the bidegree $(2,1)$. The primitive decomposition of forms reads as
\begin{equation*}
A^{2,1}=P^{2,1}\oplus L\left(A^{1,0}\right).
\end{equation*}
Passing to Bott-Chern harmonic forms, it is immediate to see that 
\begin{equation}\label{eq-incl-bc-21}
\H^{2,1}_{BC}\supseteq\left(\H^{2,1}_{BC}\cap P^{2,1}\right)\oplus L\left(\H^{1,0}_{BC}\right).
\end{equation}
However, for Aeppli harmonic forms, a similar inclusion does not hold, because
in general
$$
L\left(\H^{1,0}_{A}\right)\nsubseteq \H^{2,1}_{A}\,.
$$
Indeed, let $\alpha\in\H^{1,0}_{A}$. For bidegree reasons, $\del\delbar\alpha=0$ and $\delbar*\alpha=0$ or, equivalently, since $\alpha$ is primitive, 
$\del\delbar\alpha=0$ and $\omega^2\wedge\delbar\alpha=0$. Note that this does not imply, in general, that $L\alpha=\omega\wedge \alpha\in\H^{2,1}_{A}$, indeed we cannot conclude that $\del*(\omega\wedge\alpha)=-i\omega\wedge\del\alpha$ is equal to zero.\\
Therefore, we will focus only on Bott-Chern harmonic forms.
At this point, we could hope that the inclusion of \eqref{eq-incl-bc-21} is indeed an identity.
In fact, it does not happen, as it is shown by the following

\begin{prop}
There exists a compact almost K\"ahler $6$-dimensional manifold $(X,J,\omega)$ such that
\[
\H^{2,1}_{BC}\neq\left(\H^{2,1}_{BC}\cap P^{2,1}\right)\oplus L\left(\H^{1,0}_{BC}\right)
\]
\end{prop}
\begin{proof}
We refer to Example \ref{ex:iwasawa} for the proof of this Proposition.
\end{proof}

\begin{ex}\label{ex:iwasawa}
Let $X:=\Z[i]^3\backslash (\C^3,\cdot)$ be the Iwasawa manifold, where the group structure on $\C^3$ is defined by
\[
(w^1,w^2,w^3)\cdot(z^1,z^2,z^3)=(w^1+z^1,w^2+z^2,w^3+w^1z^2+z^3).
\]
The standard complex structure of $\C^3$ induces, on $X$, the complex structure given by
\[
\psi^1=dz^1,\ \ \ \psi^2=dz^2,\ \ \ \psi^3=-z^1dz^2+dz^3
\]
being a global coframe of $(1,0)$-forms. The complex structure equations are
\[
d\psi^1=0,\ \ \ d\psi^2=0,\ \ \ d\psi^3=-\psi^{12}.
\]
 If we set
\[
\psi^1=e^1+ie^2,\ \ \ \psi^2=e^3+ie^4,\ \ \ \psi^3=e^5+ie^6,
\]
then the real structure equations are
\[
de^1=de^2=de^3=de^4=0,\ \ \ de^5=-e^{13}+e^{24},\ \ \ de^6=-e^{14}-e^{23}.
\]
Let us consider the non integrable left-invariant almost complex structure $J$ given by
\[
\phi^1=e^1+ie^6,\ \ \ \phi^2=e^2+ie^5,\ \ \ \phi^3=e^3+ie^4
\]
being a global coframe of $(1,0)$-forms. By a direct computation the structure equations become (cf. also \cite{tardini-tomassini-dim6})
\begin{align*}
4\,d\phi^1&=-\phi^{13}-i\phi^{23}+\phi^{1\b3}+\phi^{3\b1}-i\phi^{2\b3}+i\phi^{3\b2}+\phi^{\b1\b3}-i\phi^{\b2\b3},\\
4\,d\phi^2&=-i\phi^{13}+\phi^{23}-i\phi^{1\b3}+i\phi^{3\b1}-\phi^{2\b3}-\phi^{3\b2}-i\phi^{\b1\b3}-\phi^{\b2\b3},\\
d\phi^3&=0.
\end{align*}
Endow $(X,J)$ with the left-invariant almost K\"ahler structure given by
\[
\omega=2(e^{16}+e^{25}+e^{34})=i(\phi^{1\b1}+\phi^{2\b2}+\phi^{3\b3}).
\]

First, we do the following observation that will allow us to work with only left-invariant forms (cf. \cite[Lemma 5.2]{cattaneo-tardini-tomassini}). Take $\eta\in A^{2,1}$ and assume it is left-invariant. By \eqref{eq-prim-dec-forms}, it follows that
\[
\eta=\alpha+L\beta,
\]
with $\alpha\in A^{2,1}$ primitive, i.e., $L\alpha=0$ and $\beta\in A^{1,0}$ ($\beta$ is in fact primitive for bidegree reasons). We apply $L$ and find $L\eta=L^2\beta$. Note that $L\eta$, and so $L^2\beta$, are left-invariant, and that $L^2:\Lambda^1X\to\Lambda^5X$ is an isomorphism at the level of the exterior algebra. Therefore, also $\beta$ is left-invariant. Now, since $L\beta$ and $\eta$ are left-invariant, it follows that also $\alpha$ is left-invariant. Summing up, if $\eta\in A^{2,1}$ is left-invariant, $\eta=\alpha+L\beta$ and $L\alpha=0$, then $\alpha$ and $\beta$ are left-invariant, too.

We want to find an element $\eta\in A^{2,1}$ which is contained in $\H^{2,1}_{BC}$ but is not contained in 
\[
\left(\H^{2,1}_{BC}\cap P^{2,1}\right)\oplus L\left(\H^{1,0}_{BC}\right).
\]
Thanks to the previous argument, if $\eta\in \H^{2,1}_{BC}$ is left-invariant and $\eta=\alpha+L\beta$, with $\alpha\in\H^{2,1}_{BC}\cap P^{2,1}$ and $\beta\in\H^{1,0}_{BC}$, then $\alpha$ and $\beta$ are left-invariant.

A long, but direct and straightforward computation, shows that the space of left-invariant Bott-Chern harmonic $(2,1)$-forms is
\[
\C<\phi^{13\b1}+\phi^{23\b2},\phi^{13\b2}+\phi^{23\b1}-2i\phi^{23\b2}>,
\]
while it is easy to verify that the space of left-invariant forms which are contained in $L\left(\H^{1,0}_{BC}\right)$ is
\[
\C<\phi^{13\b1}+\phi^{23\b2}>.
\]
Since $L(\phi^{13\b2}+\phi^{23\b1}-2i\phi^{23\b2})=-2iL(\phi^{23\b2})\neq0$, it means that $\phi^{13\b2}+\phi^{23\b1}-2i\phi^{23\b2}$ is not primitive. Therefore $\phi^{13\b2}+\phi^{23\b1}-2i\phi^{23\b2}$ is a left-invariant, Bott-Chern harmonic $(2,1)$-form, but it is not contained in 
\[
\left(\H^{2,1}_{BC}\cap P^{2,1}\right)\oplus L\left(\H^{1,0}_{BC}\right).
\]
\end{ex}

\end{document}